\documentclass[12pt,a4paper]{scrartcl}
\pdfoutput=1
\usepackage[utf8]{inputenc}
\usepackage[T1]{fontenc}
\usepackage{lmodern}
\usepackage[english]{babel}
\usepackage{amsmath}
\usepackage{amssymb}
\usepackage{amsthm}
\usepackage{mathtools}
\usepackage{setspace}
\usepackage{mleftright}
\usepackage{nicefrac}
\usepackage{faktor}
\usepackage[symbols,nogroupskip,sort=none]{glossaries-extra}
\usepackage[left=2.5cm,right=2.5cm,top=2.5cm,bottom=2.5cm]{geometry}
\usepackage{interval}
\intervalconfig{soft open fences}
\usepackage{xcolor}
\usepackage{hyperref}
\hypersetup{breaklinks=true, hidelinks}
\numberwithin{equation}{section}
\theoremstyle{definition}
\newtheorem{defi}{Definition}[section]
\newtheorem*{defin}{Definition}
\newtheorem{trm}[defi]{Theorem}
\newtheorem{lem}[defi]{Lemma}
\newtheorem{cor}[defi]{Corollary}
\newtheorem{prop}[defi]{Proposition}
\newtheorem*{mtrm}{Main Theorem}
\theoremstyle{definition}
\newcommand{\p}[1]{\mleft(#1\mright)}
\newcommand{\cp}[1]{\mleft\{#1\mright\}}
\newcommand{\comp}[2]{{\normalfont(cf. \cite[#1]{#2})}}
\newcommand{\comptwo}[4]{{\normalfont(cf. \cite[#1]{#2}, \cite[#3]{#4})}}
\newcommand{\inN}{\in\mathbb{N}}
\newcommand{\seq}[2]{\p{#1_{#2}}_{#2\inN}}
\newcommand{\conv}[2]{#1\xrightarrow{\ \, \ }#2}
\begin{document}
\begin{center}
\textbf{\LARGE{Gromov-Hausdorff Limits\\[0.2cm]of Closed Surfaces}}\\[0.5cm] 
\scriptsize{TOBIAS DOTT}
\end{center}
\begin{abstract}
We completely describe the Gromov-Hausdorff closure of the class of length spaces being homeomorphic to a fixed closed surface.
\end{abstract}
\let\thefootnote\relax\footnotetext{\emph{Date}: September 10, 2023.}
\let\thefootnote\relax
\let\thefootnote\relax\footnotetext{The author was supported by the DFG grant SPP 2026 (LY 95/3-2).}
\section{Introduction}
Let $M$ be a closed connected smooth manifold of dimension larger than two. Then a theorem by Ferry and Okun implies that every simply connected compact ANR carrying a length metric can be obtained as the Gromov-Hausdorff limit of length spaces being homeomorphic to $M$ \comp{p. 1866}{FO95}. In dimension two the statement does not apply. For example a space being homeomorphic to the 3-disc can not be obtained as the limit of length spaces being homeomorphic to the 2-sphere \comp{p. 269}{BBI01}. This observation leads to the following question: Is there a 2-dimensional version of the statement?\\ 
In the present paper we completely describe the Gromov-Hausdorff closure of the class of length spaces being homeomorphic to a fixed closed surface. As a corollary we derive an answer to our question.\\
We will see that the spaces of the closure have the following topological properties:
\begin{trm}\label{trm_topo} 
Let $X$ be the Gromov-Hausdorff limit of a convergent sequence of length spaces being homeomorphic to a fixed closed surface. Then the following statements apply:
\begin{itemize}
\item[1)] $X$ is at most 2-dimensional.
\item[2)] $X$ is locally simply connected.
\item[3)] There are finitely many closed surfaces $S_1,\ldots,S_n$ and $k\inN$ such that $\pi_1(X)$ is isomorphic to $\pi_1\p{S_1}\ast\ldots\ast\pi_1\p{S_n}\ast\underbrace{\mathbb{Z}\ast\ldots\ast\mathbb{Z}}_{k\text{-times}}$.
\end{itemize}
\end{trm}
\noindent
Before we state our main result, we introduce some definitions: A subset $C$ of a Peano space $X$ is called \emph{cyclicly connected} if every pair of points in $C$ can be connected by a Jordan curve in $C$. The subset is denoted as \emph{maximal cyclic} provided it is not degenerate to a point, cyclicly connected and no proper subset of a cyclicly connected subset of $X$.
\begin{defin}
We say that a Peano space is a generalized cactoid if all its maximal cyclic subsets are closed surfaces and only finitely many of them are not homeomorphic to the 2-sphere.
\end{defin}
\noindent
In the case that all maximal cyclic subsets are homeomorphic to the 2-sphere, we just call it a \emph{cactoid}.\\
A space being isometric to a metric quotient of $X$ whose underlying equivalence relation identifies exactly two points is referred to as a \emph{metric 2-point-indentification} of $X$.\\
The \emph{connectivity number} of a closed surface is defined as its first Betti number with coefficients in $\mathbb{Z}_2$. Moreover we call a surface carrying a length metric a \emph{length surface}.\\
The following main result of our work completely describes the Gromov-Hausdorff closure of the class of closed length surfaces whose connectivity number is fixed:
\begin{mtrm} Let $c\inN_0$ and $X$ be a compact length space. Then the following statements are equivalent:
\begin{itemize}
\item[1)] $X$ can be obtained as the Gromov-Hausdorff limit of closed length surfaces whose connectivity number is equal to $c$. 
\item[2)] $X$ can be obtained by a successive application of $k$ metric 2-point identifications to a geodesic generalized cactoid such that the sum of the connectivity numbers of its maximal cyclic subsets is less or equal to $c-2k$.
\end{itemize}
\end{mtrm}
\noindent
This result was partly conjectured by Young \comptwo{p. 348}{You49}{p. 854}{RS38}. Further the equivalence of the statements remains true even if we restrict the first statement to Riemannian or polyhedral 2-manifolds \comp{p. 1674}{NR23}. We will also see how the result changes if we restrict the first statement to orientable or non-orientable surfaces.\\ 
In the 1930's Whyburn already proved that the limits of length spaces being homeomorphic to the 2-sphere are cactoids \comp{p. 419}{Why35}. Moreover there is a related result about closed Riemannian 2-manifolds with uniformly bounded total absolute curvature by Shioya \comp{p. 1767}{Shi99} and a sketch of a local description of the limit spaces by Gromov \comp{S. 102}{Gro07}.\\
Finally the following corollary answers the question from the beginning:
\begin{cor}\label{cor_ANR}
Let $X$ be a metric space and $S$ be a closed surface. Then the following statements are equivalent: 
\begin{itemize}
\item[1)] $X$ is a simply connected ANR that can be obtained as the Gromov-Hausdorff limit of length spaces being homeomorphic to $S$. 
\item[2)] $X$ is a geodesic cactoid having only finitely many maximal cyclic subsets.
\end{itemize}
\end{cor}
\noindent
This work is organized as follows: We start with some preliminary notes on Gromov-Hausdorff convergence and the topology of Peano spaces, closed surfaces and gluings.\\
The third section is devoted to the topological connection between maximal cyclic subsets and their ambient space. In particular we derive a fundamental group formula for locally simply connected Peano spaces in terms of their maximal cyclic subsets. The section also provides first topological properties of generalized cactoids.\\
In Section \ref{sec} we show that the first statement of the Main Theorem implies the second. For this we begin with a consideration of sequences with additional topological control. At the end of the section we give a proof of Theorem \ref{trm_topo}.\\
The aim of the last section is to show the remaining direction of the \hyperlink{Main Theorem}{Main Theorem}. After this we prove Corollary \ref{cor_ANR}.\\
The final results of Section \ref{sec} and \ref{sec_2} refine their corresponding statement of the \hyperlink{Main Theorem}{Main Theorem}. Together they completely describe the Gromov-Hausdorff closure of the class of length spaces being homeomorphic to a fixed closed surface.  
\section{Preliminaries}
\subsection{Gromov-Hausdorff Convergence}
In this subsection we present results about Gromov-Hausdorff convergence. For basic definitions and results concerning the Gromov-Hausdorff distance we refer the reader to \cite[pp. 251-270]{BBI01}. Moreover there is a corresponding notion of convergence for maps which is treated in \cite[pp. 401-402]{Pet16}.\\
For the sake of simplicity we note the following: If we consider a Gromov-Hausdorff convergent sequence, then there are isometric embeddings of the spaces and their limit into some compact metric space such that the induced sequence Hausdorff converges to the image of the limit \comp{p. 506}{Kaw94}. Whenever we apply this statement, we will identify corresponding sets without mentioning the underlying space.\\
We introduce the concept of almost isometries: If $f\colon X\to Y$ is a map between metric spaces, we define its \emph{distortion} by:
\begin{align*}
dis(f)\coloneqq\sup_{x_1,x_2\in X}\cp{\left|d_Y\p{f\p{x_1},f\p{x_2}}-d_X\p{x_1,x_2}\right|}. \end{align*}
The map $f$ is called an \emph{$\varepsilon$-isometry} provided $dis(f)\le\varepsilon$ and $f(X)$ is an $\varepsilon$-net in $Y$.\\
The next two results are tools to prove convergence:
\begin{prop}\comp{p. 260}{BBI01} 
Let $X$ be a compact metric space and $\seq{X}{n}$ be a sequence of compact metric spaces. Then the following statements are equivalent:
\begin{itemize}
\item[1)] The sequence converges to $X$.
\item[2)] For every $n\inN$ there is an $\varepsilon_n$-isometry $f_n\colon X_n\to X$ and $\conv{\varepsilon_n}{0}$.
\end{itemize}
Moreover the equivalence remains true if we interchange $X_n$ and $X$ in the second statement.
\end{prop}
\begin{prop}\comp{p. 264}{BBI01}
Let $\seq{X}{n}$ be a sequence of compact metric spaces. Then there is a convergent subsequence if the following statements apply:
\begin{itemize}
\item[1)] There is some $D\in\mathbb{R}$ such that $diam\p{X_n}\le D$ for every $n\inN$.
\item[2)] There is a map $N\colon\interval[open]{0}{\infty}\to\mathbb{R}$ which satisfies the following property: For every $\varepsilon>0$ and $n\inN$ there is an $\varepsilon$-net in $X_n$ with at most $N(\varepsilon)$ points. 
\end{itemize}
\end{prop}
\noindent
Now we consider sequences with topological control: We say that a sequence $\seq{X}{n}$ of metric spaces is \emph{uniformly semi-locally simply connected} if $X_n$ does not contain non-contractible loops of diameter less than $2\varepsilon$ for every $n\inN$.\\
In general the fundamental group is not stable under Gromov-Hausdorff convergence. The upcoming result presents a case in which it is stable:  
\begin{trm}\comp{p. 3588}{SW01}
Let $X$ be a semi-locally simply connected compact metric space and $\seq{X}{n}$ be a uniformly semi-locally simply connected  sequence of compact length spaces. If $\conv{X_n}{X}$, then $\pi_1(X)$ is isomorphic to $\pi_1\p{X_n}$ for almost all $n\inN$.  
\end{trm}
\noindent
From \cite[p. 267]{BBI01} we deduce the following:
\begin{prop}
Every compact metric tree can be obtained as the limit of finite metric trees. 
\end{prop}
\noindent
We note that the class of compact metric trees is equal to the class of compact length spaces not containing Jordan curves \comp{pp. 367-368}{DK18}.\\
Finally we state \hypertarget{Whyburn's theorem}{Whyburn's theorem} about the limits of 2-spheres:
\begin{trm}\label{trm_why}\comp{p. 419}{Why35} The limits of length spaces being homeomorphic to the 2-sphere are cactoids.
\end{trm}
\subsection{Topology of Peano Spaces}
This subsection is devoted to Peano Spaces. Our main source about this topic is \cite{Why42}.\\
A compact connected locally connected metric space is called a \emph{Peano space}.\\
From a topological point of view there is no difference between Peano spaces and compact length spaces:
\begin{trm}\label{trm_Peano}\comp{p. 1109}{Bin49}
Every Peano space is homeomorphic to a compact length space. 
\end{trm}
\noindent
A point $x$ of a connected metric space $X$ is denoted as a \emph{cut point} of $X$, if $X\setminus\cp{x}$ is disconnected. Moreover we say that $x$ is a \emph{local cut point} of $X$ provided there is some connected open neighborhood $U$ of $x$ such that $U\setminus\cp{x}$ is disconnected.\\ 
The following result summarizes basic properties of Peano spaces:
\begin{lem}\label{lem_topo}\comp{pp. 65-71, 143}{Why42} Let $X$ be a Peano space space and $T$ be a maximal cyclic subset of $X$. Then the following statements apply:
\begin{itemize}
\item[1)] There are only countably many maximal cyclic subsets in $X$.
\item[2)] There are only countably many connected components in $X\setminus T$.
\item[3)] $T$ contains only countably many cut points of $X$. 
\item[4)] If $\seq{C}{n}$ is a sequence of pairwise distinct connected components of $X\setminus T$, then $\conv{diam\p{C_n}}{0}$.
\item[5)] If $\seq{T}{n}$ is a sequence of pairwise distinct maximal cyclic subsets of $X$, then $\conv{diam\p{T_n}}{0}$. 
\item[6)] If $C$ is a connected component of $X\setminus T$, then there is some $x\in T$ such that $\partial C=\cp{x}$.
\item[7)] The map $r\colon X\to T$ such that $r(x)=x$ for every $x\in T$ and $r(x)\in\partial C$ for every connected component $C$ of $X\setminus T$ and $x\in C$ is continuous. 
\end{itemize}
\end{lem}
\subsection{Topology of Closed Surfaces}
We want to fix two results concerning closed surfaces. The first one describes a property which is shared by their fundamental groups:
\begin{prop}\label{prop_free}\comptwo{pp. 141, 264}{Loe17}{pp. 1480-1481}{Bab95} Let $S$ be a closed surface. Then $\pi_1(S)$ is not isomorphic to a free product of non-trivial groups.
\end{prop}
\noindent
Finally we have the following classification of non-contractible Jordan curves in closed surfaces:
\begin{prop}\label{prop_jord}\comp{pp. 54-55}{MST16}
Let $S$ be a closed surface of connectivity number $c$ and $J$ be a non-contractible Jordan curve in $S$. Then the topological quotient $X\coloneqq\faktor{S}{J}$ can be described in one of the following ways:
\begin{itemize}
\item[1)] There are $c_1,c_2\inN$ with $c_1+c_2=c$, a closed surface $S_1$ of connectivity number $c_1$ and a closed surface $S_2$ of connectivity number $c_2$ such that $X$ is a wedge sum of $S_1$ and $S_2$. Moreover at least one of the surfaces in non-orientable if and only if $S$ is non-orientable.
\item[2)] There is a closed surface of connectivity number $c-2$ such that $X$ is a topological 2-point identification of it. Moreover the surface is orientable if $S$ is orientable.
\item[3)] $X$ is a closed surface of connectivity number $c-1$ and $S$ is non-orientable. 
\end{itemize}
\end{prop}
\subsection{Fundamental Groups of Gluings}
We present two results which are related to the Seifert-Van Kampen Theorem.
\begin{prop}\label{prop_SVK}\comp{p. 176}{Gri54} Let $X$ and $Y$ be locally simply connected and path connected metric spaces. Then the fundamental group of a wedge sum of $X$ and $Y$ is isomorphic to $\pi_1(X)\ast\pi_1(Y)$.   
\end{prop}
\noindent
Since metric 2-point identifications play an important role in our investigation, the next result provides a useful tool. It follows from the HNN-Seifert Van Kampen Theorem in \cite{Fri22}[p. 1292-1293].
\begin{prop}\label{prop_HNN}
Let $X$ be a locally-simply connected and path-connected metric space. Then the following statements apply:
\begin{itemize}
\item[1)] If $Y$ is a topological 2-point identification of $X$, then the fundamental group $\pi_1(Y)$ is isomorphic to $\pi_1(X)\ast\mathbb{Z}$. 
\item[2)] If $X$ contains a local cut point not being a cut point, then there is a group $G$ such that the fundamental group $\pi_1(X)$ is isomorphic to $G\ast\mathbb{Z}$. 
\end{itemize}
\end{prop}
\glsxtrnewsymbol[description={The class of compact metric spaces.}]{1}{$\mathcal{M}$}
\glsxtrnewsymbol[description={The class of closed length surfaces whose connectivity number is equal to $c$.}]{2}{$\mathcal{S}(c)$}
\glsxtrnewsymbol[description={The class of spaces in $\mathcal{S}(c)$ that do not contain non-contractible loops of diameter less than $2\varepsilon$.}]{3}{$\mathcal{S}\p{c,\varepsilon}$}
\glsxtrnewsymbol[description={The class of geodesic generalized cactoids such that the connectivity numbers of their maximal cyclic subsets sum up to $c$}.]{4}{$\mathcal{G}(c)$}
\glsxtrnewsymbol[description={The class of successive metric wedge sums of non-degenerate cyclicly connected compact length spaces and finite metric trees.}]{5}{$\mathcal{W}$}
\glsxtrnewsymbol[description={The class of successive metric wedge sums of closed length surfaces such that every wedge point is only shared by two of their surfaces.}]{6}{$\mathcal{W}_0$}
\printunsrtglossary[type=symbols,style=long,title=Notation]
\noindent
We note that we allow a change of the wedge point in every construction step of a successive metric wedge sum.
\section{Topology via Maximal Cyclic Subsets}\label{section_limits}
\noindent
In this section we investigate the topological connection between maximal cyclic subsets and their ambient space. From the results we derive first topological properties of generalized cactoids.
\subsection{A Dimension Bound}
First we consider the dimension of Peano spaces. Throughout this work the term dimension refers to the covering dimension.\\
There are many properties which are satisfied by the whole Peano space provided they are shared by all its maximal cyclic subsets \comp{pp. 81-83 }{Why42}. The next result contains an example of such a property:
\begin{prop}
Let $X$ be a Peano space. Then the following statements apply:
\begin{itemize}
\item[1)] If all maximal cyclic subsets of $X$ are at most n-dimensional, then $X$ is at most n-dimensional. \comp{p. 82}{Why42}
\item[2)] If $X$ is at most $n$-dimensional and $Y$ is a metric 2-point identification of $X$, then $Y$ is at most $n$-dimensional \comp{pp. 266, 271}{Sak13}.
\end{itemize}
\end{prop}
\noindent
As a consequence we derive the following result about generalized cactoids:
\begin{cor}\label{cor_dim}
Let $X$ be a space that can be obtained by a successive application of metric 2-point identifications to a generalized cactoid. Then $X$ is at most 2-dimensional.    
\end{cor}
\subsection{Local Contractibility}
Now we consider the property of local contractibility.\\
A metric space $X$ is called \emph{locally strongly contractible} provided every $x\in X$ has arbitrarily small open neighborhoods such that $\cp{x}$ is a strong deformation retract of them. For example $X$ is locally strongly contractible if it is a manifold. 
\begin{prop}
Let $X$ be a Peano space. Then the following statements apply:
\begin{itemize}
\item[1)] If $X$ contains infinitely many non-contractible maximal cyclic subsets, then $X$ is not locally contractible.
\item[2)] If $X$ has only finitely many maximal cyclic subsets and all of them are locally strongly contractible, then $X$ is locally strongly contractible. 
\end{itemize}
\end{prop}
\begin{proof}
1) There is a sequence of pairwise distinct non-contractible maximal cyclic subsets in $X$. From Lemma \ref{lem_topo} follows that there is a subsequence converging to some $p\in X$. Hence an arbitrarily small open neighborhood of $p$ contains a non-contractible maximal cyclic subset of $X$.\\
For the sake of contradiction we assume that $X$ is locally contractible. Then there is a contractible open neighborhood $U$ of $p$. This neighborhood contains some non-contractible maximal cyclic subset $T$ of $X$. By Lemma \ref{lem_topo} we have that $T$ is a retract of $U$. Therefore $T$ is also contractible. A contradiction.\\
2) We denote the number of maximal cyclic subsets in $X$ by $n$.\\
If $n=0$, then $X$ is homeomorphic to a compact metric tree and hence locally strongly contractible \comp{p. 20}{AKP19}.\\
If $n=1$, we consider $p\in X$ and $\varepsilon>0$. We may assume that $p$ is contained in the only maximal cyclic subset $T$ of $X$. Then there is an open neighborhood $U_0$ of $p$ in $T$ with diameter less than $\varepsilon$ such that $\cp{p}$ is a strong deformation retract of $U_0$. Further we may assume that $U_0$ contains infinitely many cut points of $X$.\\
Let $\seq{c}{k}$ be an enumeration of these cut points. We write $C_k$ for the closure of the union of all connected components of $X\setminus T$ whose boundaries are given by $\cp{c_k}$. Then $C_k$ is a compact metric tree. Hence there is an open neighborhood $U_k$ of $c_k$ in $C_k$ with diameter less than $\varepsilon$ such that $\cp{c_k}$ is a strong deformation retract of $U_k$. In particular there is a homotopy $H_k$ between the identity map on $U_k$ and the map sending every point of $U_k$ to $c_k$ such that $H_k\p{c_k,t}=c_k$ for every $t\in\interval{0}{1}$.\\ 
Now we set $U\coloneqq \cup_{k\inN_0}U_k$ and define a map $H\colon U\times\interval{0}{1}\to U$ by:
\begin{align*}
H\p{x,t}=\left\{ \begin{array}{lr}
H_k\p{x,t},&x\in U_k;\\
x,&x\in U_0. 
\end{array}\right.
\end{align*}
The diameter of $U$ is less than $2\varepsilon$. From Lemma \ref{lem_topo} we derive that $U$ is an open neighborhood of $p$ in $X$ and $H$ is continuous. Hence $U_0$ is a strong deformation retract of $U$ and it follows that $\cp{p}$ is a strong deformation retract of $U$. We conclude that $X$ is locally strongly contractible.\\
If $n\ge 2$, then $X$ is homeomorphic to a wedge sum of Peano spaces with less than $n$ maximal cyclic subsets. If both spaces are locally strongly contractible, then so is $X$. Hence the claim follows by induction.  
\end{proof}
\noindent
A metric space whose dimension is finite is an ANR if and only if it is locally contractible \comp{p. 347, 392}{Sak13}. Hence Corollary \ref{cor_dim} and the last proposition imply the following statement:
\begin{cor}\label{cor_gen_ANR}
Let $X$ be a generalized cactoid. Then the following statements are equivalent:
\begin{itemize}
\item[1)] $X$ is an ANR.
\item[2)] $X$ has only finitely many maximal cyclic subsets.
\end{itemize}
\end{cor}
\subsection{A Fundamental Group Formula}
In this subsection we present a formula for the fundamental group of a locally simply connected Peano space in terms of its maximal cyclic subsets. For this we first reduce the complexity of the problem: 
\begin{lem}\label{lem_finite}
Let $X$ be a compact length space and $\p{T_k}^{\infty}_{k=1}$ be an enumeration of its maximal cyclic subsets. Then $X$ can be obtained as the limit of compact length spaces having only finitely many maximal cyclic subsets. Further the maximal cyclic subsets of the space with index $n$ are in isometric one-to-one correspondence with $\cp{T_k}^{n}_{k=1}$ for every $n\inN$.
\end{lem}
\begin{proof}
Let $n,k\inN$. We define an equivalence relation $\sim$ on $X$ as follows: $x\sim y$ if and only if $x$ and $y$ lie in the same connected component of $\cup_{m=n+1}^{n+k}T_m$. Further we define $X_{n,k}$ as the metric quotient $\faktor{X}{\sim}$ and denote the corresponding projection map by $p_{n,k}$. Then $p_{n,k}$ is surjective and 1-lipschitz. Hence we may assume that there is a space $X_n\in\mathcal{M}$ and a map $p_n\colon X\to X_n$ such that $\conv{X_{n,k}}{X_n}$ and $\conv{p_{n,k}}{p_n}$ uniformly.\\
Since the map $p_{n,k}$ is monotone and its restriction to $T_m$ is distance preserving for every $m\in\cp{1,\ldots,n}$, the same applies to $p_n$ \comp{p. 174}{Why42}.\\ Let $T$ be a maximal cyclic subset of $X_n$. Then we find some $k\inN$ such that $T\subset p_n\p{T_k}$ \comp{p. 145-146}{Why42}. Because $p_n$ is constant on $T_m$ for every $m\inN$ with $m>n$, we have $k\le n$. Hence $p_n\p{T_k}$ is cyclicly connected and we derive that the inclusion is an equality.\\ Due to the fact that for every non-degenerate cyclicly connected subset of $X_n$ there is a maximal cyclic subset containing it \comp{p. 79}{Why42}, we derive that $\cp{p_n\p{T_k}}^n_{k=1}$ is the set of maximal cyclic subsets of $X_n$ and has cardinality $n$.\\ Since $p_n$ is surjective and 1-lipschitz, we may assume that there is $\tilde{X}\in\mathcal{M}$ with $\conv{X_n}{\tilde{X}}$. Choosing a diagonal sequence we may assume that $\conv{X_{n,n}}{\tilde{X}}$ and $\p{p_{n,n}}_{n\inN}$ is uniformly convergent. Finally the limit map is an isometry between $X$ and $\tilde{X}$.
\end{proof}
\begin{lem}
Let $X$ be a compact length space having only finitely many maximal cyclic subsets. Then $X$ can be obtained as the limit of spaces in $\mathcal{W}$. Further the maximal cyclic subsets of the spaces of the sequence are in isometric one-to-one correspondence with those of $X$.
\end{lem}
\begin{proof}
Let $\varepsilon$ be the minimum of the diameters of the maximal cyclic subsets of $X$. If $n\inN$ and $T$ is a maximal cyclic subset of $X$, we denote the set of connected components of $X\setminus T$ having diameter less than $\nicefrac{\varepsilon}{n}$ by $\mathcal{C}_T$. Moreover we set $\mathcal{T}$ as the set of maximal cyclic subsets of $X$ and $X_n\coloneqq X\setminus\p{\cup_{T\in \mathcal{T}}\cup_{C\in\mathcal{C}_T}C}$.\\
The space $X_n$ is a compact length space which has the same maximal cyclic subsets as $X$. Further it is an $\nicefrac{\varepsilon}{n}$-net in $X$ and it follows that the inclusion map from $X_n$ to $X$ is an $\nicefrac{\varepsilon}{n}$-isometry. Hence we have $\conv{X_n}{X}$.\\
Moreover $X_n$ is a successive metric wedge sum of its maximal cyclic subsets and finitely many compact metric trees. Let $D$ be such a tree. We denote the wedge points lying in $D$ by $p_1,\ldots,p_N$. There is a sequence $\p{D_k}_{k\inN}$ of finite metric trees converging to $D$. Further we may assume the existence of a sequence $\seq{f}{k}$ such that $f_k$ is a $\nicefrac{1}{k}$-isometry from $D$ to $D_k$. We define $X_{n,k}$ as the successive metric wedge sum created by replacing $D$ with $D_k$ in $X_n$ and the corresponding wedge points with $f_k\p{p_1},\ldots,f_k\p{p_N}$.\\
This new space is again a compact length space such that its maximal cyclic subsets are in isometric one-to-one correspondence with those of $X$. In particular we may assume that $X_{n,k}\in\mathcal{W}$. Otherwise we repeat the argument above until every tree is replaced by a finite one. Moreover we have that $\conv{X_{n,k}}{X_n}$ since $f_k$ defines a $\nicefrac{1}{k}$-isometry between $X_n$ and $X_{n,k}$. Choosing a diagonal sequence we may assume that $\conv{X_{n,n}}{X}$.
\end{proof}
\begin{cor}\label{cor_finite}
Let $X$ be a compact length space and $\p{T_k}^{\infty}_{k=1}$ be an enumeration of its maximal cyclic subsets. Then $X$ can be obtained as the limit of spaces in $\mathcal{W}$. Further the maximal cyclic subsets of the space with index $n$ are in isometric one-to-one correspondence with $\cp{T_k}^{n}_{k=1}$ for every $n\inN$. 
\end{cor}
\noindent 
Every compact interval or point can be obtained as the limit of spaces in $\mathcal{S}(0)$. Hence we can proceed as in the second half of the last proof to get rid of the finite metric trees first and then the cut points lying in more than one maximal cyclic subset. From this we derive the following statement for later purposes:
\begin{cor}\label{cor_spher}
Let $X$ be geodesic generalized cactoid and $\p{T_k}^{\infty}_{k=1}$ be an enumeration of its maximal cyclic subsets. Then $X$ can be obtained as the limit of spaces in $\mathcal{W}_0$. Further the maximal cyclic subsets of the space with index $n$ are in isometric one-to-one correspondence with $\cp{T_k}^{n}_{k=1}$ and finitely many spaces in $\mathcal{S}(0)$ for every $n\inN$. 
\end{cor}
\noindent
Now we are able to state the formula:
\begin{prop}\label{prop_fund}
Let $X$ be a locally simply connected Peano space and  $\p{T_n}^{\infty}_{n=1}$ be an enumeration of its maximal cyclic subsets. Then $\pi_1(X)$ is isomorphic to $\pi_1\p{T_1}\ast\ldots\ast\pi_1\p{T_n}$ for almost all $n\inN$.
\end{prop}
\begin{proof}
Since $X$ is locally simply connected, the same applies to its maximal cyclic subsets. Moreover there is some $\varepsilon>0$ such that every loop in $X$ of diameter less than $\varepsilon$ lies in some simply connected subset of $X$. If $T$ is a maximal cyclic subset of $X$, then also every loop in $T$ of diameter less than $\varepsilon$ lies in some simply connected subset of $T$.\\
By Theorem \ref{trm_Peano} we may assume that $X$ is geodesic. Let $\seq{X}{n}$ be a sequence as in Corollary \ref{cor_finite}. An application of Proposition \ref{prop_SVK} yields that the sequence is uniformly semi-locally simply connected. Hence $\pi_1(X)$ is isomorphic to $\pi_1\p{X_n}$ for almost all $n\inN$. From the same proposition we also derive that $\pi_1\p{X_n}$ is isomorphic to $\pi_1\p{T_1}\ast\ldots\ast\pi_1\p{T_n}$ for every $n\inN$. This closes the proof. 
\end{proof}
\subsection{Cactoids}
As it will turn out limits of spaces in $\mathcal{S}(c)$ locally look like cactoids (see Corollary \ref{cor_loc_cact}). Hence we are interested in the topology of cactoids. Again we first reduce the complexity of the problem:
\begin{lem}\label{lem_round}
Let $X\in\mathcal{G}(0)$. Then $X$ is homeomorphic to a space that can be obtained as the limit of compact length spaces whose maximal cyclic subsets are isometric to round 2-spheres.
\end{lem}
\begin{proof}
First we may assume that there are infinitely many maximal cyclic subsets in $X$. Let $\p{T_n}^{\infty}_{n=1}$ be an enumeration of them. There is a homeomorphism $f_n$ from $T_n$ to the round 2-sphere of diameter $\nicefrac{1}{2^n}$. We denote this 2-sphere by $S_n$.\\
The following condition naturally induces an equivalence relation $\sim$ on the  disjoint union of the closures of the connected components of $X\setminus T_1$ and $S_1$: $x\sim y$ if $f_1(x)=y$. We equip the disjoint union with its induced length metric and denote the corresponding metric quotient by $Y_1$. Analogously we construct the space $Y_2$ using the disjoint union of the closures of the connected components of $Y_1\setminus T_2$ and $S_2$. We continue this construction and derive a sequence $\seq{Y}{n}$ of metric spaces.\\
We note that there is an induced enumeration of the maximal cyclic subsets of $Y_n$. If $k\inN$, we define the metric space $Y_{n,k}$ in the same way for $Y_n$ as $X_{n,k}$ is defined for $X$ in the proof of Lemma \ref{lem_finite}.\\
The maps $f_1,\ldots,f_n$ naturally induce a homeomorphism from $X$ to $Y_n$. We denote the composition of this homeomorphism with the projection map from $Y_n$ to $Y_{n,k}$ by $g_{n,k}$. Moreover we find a map $p_{n,k}\colon Y_{n+1,k}\to Y_{n,k+1}$ such that the diagram together with the maps $g_{n+1,k}$ and $g_{n,k+1}$ commutes.  This map is 1-lipschitz and has a distortion less than $\nicefrac{1}{2^n}$. Furthermore the sequence $\p{g_{n,k}}_{k\inN}$ is uniformly equicontinuous. We may assume that there is a space $\tilde{Y}_{n}$ such that $\conv{Y_{n,k}}{\tilde{Y}_n}$. In addition we may assume the existence of maps $p_n\colon \tilde{Y}_{n+1}\to \tilde{Y}_{n}$ and $g_n\colon X\to\tilde{Y}_n$ such that $\conv{p_{n,k}}{p_n}$ and $\conv{g_{n,k}}{g_n}$ uniformly \comp{p. 402}{Pet16}.\\
Using the proof of Lemma \ref{lem_finite} we see that every maximal cyclic subset of $\tilde{Y}_n$ is isometric to a round 2-sphere. Further the diagram consisting of $p_n$, $g_n$ and $g_{n+1}$ commutes and $p_n$ is a 1-lipschitz map that has a distortion less or equal to $\nicefrac{1}{2^n}$. If we set $\varepsilon_k\coloneqq\sum^{\infty}_{m=k}\nicefrac{1}{2^m}$, then $p_k\circ\ldots\circ p_{n-1}$ is an $\varepsilon_k$-isometry between $\tilde{Y}_n$ and $\tilde{Y}_k$ for every $n\inN$ with $n>k$. Hence $\seq{\tilde{Y}}{n}$ is convergent \comp{p. 399}{Pet16} and we denote its limit space by $Y$. Moreover $\seq{g}{n}$ is uniformly equicontinuous and we may assume the existence of a map $g\colon X\to Y$ such that $\conv{g_n}{g}$ uniformly. Finally $g$ is bijective and hence a homeomorphism.  
\end{proof} 
\noindent
\begin{cor}\label{cor_simply}
Cactoids are locally simply connected and simply connected.
\end{cor}
\begin{proof}
In a successive metric wedge sum of round 2-spheres and finite metric trees is every open ball simply connected. A theorem by Petersen implies that this property is stable under Gromov-Hausdorff convergence \comp{p. 501}{Pet93}. Finally the claim follows by Corollary \ref{cor_finite} and Lemma \ref{lem_round}.
\end{proof}
\section{The Limit Spaces}\label{sec}
The goal of this chapter is to show that the first statement of the \hyperlink{Main Theorem}{Main Theorem} implies the second. As a consequence we derive Theorem \ref{trm_topo}. 
\subsection{Controlled Convergence}
First we investigate the limit spaces of sequences with additional topological control.\\
A metric space $X$ is called a \emph{local cactoid} if every point in $X$ has an open neighborhood being homeomorphic to an open subset of a cactoid.
\begin{lem}\label{lem_local}
Let $\seq{X}{n}$ be a sequence in $\mathcal{S}\p{c,\varepsilon}$ and $X\in\mathcal{M}$ with $\conv{X_n}{X}$. Then $X$ is a local cactoid. \end{lem}
\begin{proof}
First we may assume that all surfaces of the sequence are not homeomorphic to the 2-sphere. Otherwise the claim follows by Theorem \ref{trm_why}.\\
Let $x\in X$. Then there is a sequence $\seq{x}{n}$ with $x_n\in X_n$ such that $\conv{x_n}{x}$. Every Jordan curve in $B_{\varepsilon}\p{x_n}$ is contractible in $X_n$ and hence bounds a topological closed disc \comp{p. 85}{Eps66}. We set $A$ as the union of all these discs and $B_{\varepsilon}\p{x_n}$.\\ It follows that $A$ is open and connected. Further every Jordan curve in the ball is contractible in $A$. In particular the same holds for all loops in the ball \comp{p. 626}{LW18}. Since every loop in $A$ admits a finite subdivision into paths lying in the ball or one of the corresponding open discs, an induction yields that every loop in $A$ is homotopic to some loop in the ball. Therefore $A$ is simply connected and we conclude that $A$ is homeomorphic to the plane or the 2-sphere \comp{p. 85}{Eps66}. Because $X_n$ is not homeomorphic to the 2-sphere, the first case applies. We derive that the metric quotient $Y_n\coloneqq\faktor{X_n}{A^c}$ is homeomorphic to the 2-sphere. Especially we have $Y_n\in\mathcal{S}(0)$.\\
Since the natural projection $p_n\colon X_n\to Y_n$ is surjective and 1-lipschitz, we may assume that there is a space $Y\in\mathcal{M}$ and a map $p\colon X\to Y$ such that $\conv{Y_n}{Y}$ and $\conv{p_n}{p}$. From Theorem \ref{trm_why} we get that $Y$ is a cactoid. Further we may assume that the sequences $\p{\bar{B}_{\frac{\varepsilon}{2}}\p{x_n}}_{n\inN}$ and $\p{\bar{B}_{\frac{\varepsilon}{2}}\p{p_n\p{x_n}}}_{n\inN}$ are convergent. Their limits are given by $\bar{B}_{\frac{\varepsilon}{2}}(x)$ and $\bar{B}_{\frac{\varepsilon}{2}}\p{p(x)}$. Because $p_n$ defines an isometry between $\bar{B}_{\frac{\varepsilon}{2}}\p{x_n}$ and $\bar{B}_{\frac{\varepsilon}{2}}\p{p_n\p{x_n}}$, the same applies to $p$ with respect to $\bar{B}_{\frac{\varepsilon}{2}}(x)$ and $\bar{B}_{\frac{\varepsilon}{2}}\p{p(x)}$. In particular this also holds for the corresponding open balls. We finally conclude that $X$ is a local cactoid.
\end{proof}
\begin{cor}\label{cor_fund_cyc}
Let $\seq{X}{n}$ be a sequence in $\mathcal{S}\p{c,\varepsilon}$ and $X\in\mathcal{M}$ with $\conv{X_n}{X}$. Then the following statements apply:
\begin{itemize}
\item[1)] $X$ is locally simply connected.
\item[2)] If $c=0$, then the maximal cyclic subsets of $X$ are simply connected.
\item[3)] If $c>0$, then there is a closed surface $S$ of connectivity number $c$ such that the fundamental group of one maximal cyclic subset of $X$ is isomorphic to $\pi_1(S)$ and all other maximal cyclic subsets are simply connected. Moreover $S$ is orientable if and only if $X_n$ is orientable for infinitely many $n\inN$.
\end{itemize}
\end{cor}
\begin{proof}
Combining Corollary \ref{cor_simply} and Lemma \ref{lem_local} we derive the first statement. Since the sequence is uniformly semi-locally simply connected, it follows that $\pi_1(X)$ isomorphic to $\pi_1\p{X_n}$ for almost all $n\inN$. Hence Proposition \ref{prop_free} and Proposition \ref{prop_fund} close the proof.
\end{proof}
\noindent
From the upcoming lemma follows that limits of spaces in $\mathcal{S}\p{c,\varepsilon}$ are generalized cactoids:
\begin{lem}
Let $\seq{X}{n}$ be a sequence in $\mathcal{S}\p{c,\varepsilon}$, $X\in\mathcal{M}$ with $\conv{X_n}{X}$ and $T$ be a maximal cyclic subset of $X$. Then $T$ is a closed surface.
\end{lem}
\begin{proof}
First we show that $T$ is free of local cut points: For the sake of contradiction we assume that $T$ contains a local cut point. Due to  Corollary \ref{cor_fund_cyc} the space $X$ is locally simply connected. From Proposition \ref{prop_HNN} follows that there is a group $G$ such that $\pi_1(T)$ is isomorphic to $G\ast \mathbb{Z}$. Moreover Corollary \ref{cor_fund_cyc} yields that $\pi_1(T)$ is isomorphic to the fundamental group of some closed surface. This contradicts Proposition \ref{prop_free}.\\
Let $p\in T$. Then Lemma \ref{lem_local} implies that there is a connected open neighborhood $V$ of $p$ in $X$ and a homeomorphism $f$ from $V$ to an open subset of some cactoid $C$. Further Lemma \ref{lem_topo} yields that $T\cap V$ is connected. Since $T$ is free of local cut points, there is a Jordan curve $J$ in $V\cap T$. In particular $f(J)$ is contained in some maximal cyclic subset $S$ of $C$.\\
The subset $f(V)\cap S$ is connected and $S$ is free of local cut points. Therefore we derive that $f\p{V\cap T}=f(V)\cap S$. Hence $V\cap T$ is homeomorphic to an open subset of the 2-sphere.\\
We conclude that $T$ is a surface. Especially $T$ is a closed surface because  $\pi_1(T)$ is isomorphic to the fundamental group of some closed surface.
\end{proof}
\noindent
Combining the last lemma and Corollary \ref{cor_fund_cyc} we get the following result:
\begin{cor}\label{cor_simply_limit}
Let $\seq{X}{n}$ be a sequence in $\mathcal{S}\p{c,\varepsilon}$, where $c>0$, and $X\in\mathcal{M}$ with $\conv{X_n}{X}$. Then one maximal cyclic subset $T$ of $X$ is a closed surface of connectivity number $c$ and all other maximal cyclic subsets are homeomorphic to the 2-sphere. Moreover $T$ is orientable if and only if $X_n$ is orientable for infnitely many $n\inN$.
\end{cor}
\subsection{The General Case}
Next we see what happens if we omit the additional topological control.\\
If a sequence in $\mathcal{S}(c)$ has no uniformly semi-locally simply connected subsequence, then there is a subsequence and a sequence of Jordan curves as in the following lemma:
\begin{lem}\label{lem_jordan}
Let $\seq{X}{n}$ be a sequence in $\mathcal{S}(c)$ and $X\in\mathcal{M}$ with $\conv{X_n}{X}$. Further let $\seq{J}{n}$ be a sequence such that $J_n$ is a non-contractible Jordan curve in $X_n$ for every $n\inN$ and $\conv{diam\p{J_n}}{0}$. Then one of the following cases applies:
\begin{itemize}
\item[1)] There are $c_1,c_2\inN$ with $c_1+c_2=c$, a convergent sequence of spaces in $\mathcal{S}\p{c_1}$ and a convergent sequence of spaces in $\mathcal{S}\p{c_2}$ such that $X$ is a metric wedge sum of their limits. If $X_n$ is non-orientable for infinitely many $n\inN$, then the surfaces of at least one of the sequences may be chosen to be non-orientable.  
\item[2)] There is a convergent sequence of spaces in $\mathcal{S}\p{c-2}$ such that $X$ is its limit or a metric 2-point identification of it.  
\item[3)] There is a sequence of spaces in $\mathcal{S}\p{c-1}$ converging to $X$.
\end{itemize}
If $X_n$ is orientable for infinitely many $n\inN$, then always one of the first two cases applies and the surfaces of the corresponding sequences may be chosen to be orientable.
\end{lem}
\begin{proof}
First we may assume that all surfaces of the sequence are orientable or all are non-orientable. Moreover we may assume that the Jordan curves all belong to the same class in the sense of Proposition \ref{prop_jord}. In particular we have that the corresponding sequence $\seq{Y}{n}$ of metric quotients is convergent with limit $X$. We just go through the cases of the proposition:\\
In the first case there are $c_1,c_2\inN$ with $c_1+c_2=c$ such that $Y_n$ is a metric wedge sum of a space in $\mathcal{S}\p{c_1}$ and a space in $\mathcal{S}\p{c_2}$. Especially at least one of the surfaces is non-orientable if and only if $X_n$ is non-orientable. Moreover we may assume that the sequence of the wedge points and the sequences of the surfaces considered as subsets of the wedge sums are convergent. We derive that $X$ is a metric wedge sum of the limits \comp{p. 412}{Why35}.\\
Now we consider the second case: Then there is a space $Z_n$ in $\mathcal{S}\p{c-2}$ such that $Y_n$ is a metric 2-point identification of $Z_n$. In particular $Z_n$ is orientable if $X_n$ is orientable. Further we may assume the sequence $\seq{Z}{n}$ to be convergent. It follows that $X$ is the limit or a metric 2-point identification of it.\\
If we look at the third case, then we  have that $Y_n$ is a space in $\mathcal{S}\p{c-1}$ and $X_n$ is non-orientable.
\end{proof}
\noindent
The final result of this section refines a statement of the \hyperlink{Main Theorem}{Main Theorem}: 
\begin{trm}\label{trm_first_Main} 
Let $\seq{X}{n}$ be a sequence in $\mathcal{S}(c)$ and $X\in\mathcal{M}$ with $\conv{X_n}{X}$. Then $X$ can be obtained by a successive application of $k$ metric 2-point identifications to some space $Y\in\mathcal{G}\p{c_0}$ where $c_0\le c-2k$. Moreover the following statements apply: 
\begin{itemize}
\item[1)] If $X_n$ is orientable for infinitely many $n\inN$, then the maximal cyclic subsets of $Y$ are orientable. 
\item[2)] If $X_n$ is non-orientable for infinitely many $n\inN$, the maximal cyclic subsets of $Y$ are orientable and $k=0$, then $c_0<c$. \end{itemize} 
\end{trm}
\begin{proof}
The proof proceeds by induction over the connectivity number:\\
In the case $c=0$ the claim directly follows by Theorem \ref{trm_why}.\\
Now we consider the case $c>0$. Moreover we assume that the claim is true, if the connectivity number is less than $c$. Provided the sequence is uniformly semi-locally simply connected, the claim directly follows by Corollary \ref{cor_simply_limit}. Otherwise we may assume that there is a sequence $\seq{J}{n}$ such that $J_n$ is a non-contractible Jordan curve in $X_n$ and $\conv{diam\p{J_n}}{0}$. Hence one of the cases of Lemma \ref{lem_jordan} applies. We note that the surfaces of the sequences occurring there have a connectivity number less than $c$.\\
Finally an application of the induction hypothesis and the following observation yield the claim:
Let $Y_1$ and $Y_2$ be metric spaces. Further let $Z_i$ be a space that can be obtained by a successive application of $k_i$ metric 2-point identifications to $Y_i$. Then every metric wedge sum of $Z_1$ and $Z_2$ is a space that can be obtained by a successive application of $k_1+k_2$ metric 2-point identifications to a metric wedge sum of $Y_1$ and $Y_2$. Moreover every metric 2-point identification of $Z_1$ is a space that can be obtained by a successive application of $k_1+1$ metric 2-point identifications to $Y_1$.
\end{proof}
\noindent 
The property of being a local cactoid is stable under applications of metric wedge sums and metric 2-point identifications. Hence the induction above also yields the following result:
\begin{cor}\label{cor_loc_cact}
Let $\seq{X}{n}$ be a sequence in $\mathcal{S}(c)$ and $X\in\mathcal{M}$ with $\conv{X_n}{X}$. Then $X$ is a local cactoid.   
\end{cor}
\noindent
Finally we derive Theorem \ref{trm_topo}:
\begin{proof}[Proof of Theorem \ref{trm_topo}]
From Corollary \ref{cor_dim} follows that $X$ is at most 2-dimensional. Further Corollary \ref{cor_simply} and Corollary \ref{cor_loc_cact} imply that $X$ is locally simply connected. Finally Proposition \ref{prop_HNN} and Proposition \ref{prop_fund} yield the desired representation of $\pi_1(X)$.    
\end{proof}
\section{Approximation of Generalized Cactoids}\label{sec_2}
The aim of this chapter is to prove that the second statement of the  \hyperlink{Main Theorem}{Main Theorem} implies the first. At the end we show Corollary \ref{cor_ANR}.\\
We start by showing that metric wedge sums of closed length surfaces can be approximated by closed length surfaces:
\begin{lem}\label{lem_wedge}
Let $S_1$ and $S_2$ be closed length surfaces and $X$ be a metric wedge sum of them. Then $X$ can be obtained as the limit of length spaces being homeomorphic to a connected sum of $S_1$ and $S_2$.
\end{lem}
\begin{proof}
We denote the wedged points by $p_1$ and $p_2$. For every $i\in\cp{1,2}$ and $n\inN$ there is a topological closed disc $D_{i,n}$ of diameter less than $\nicefrac{1}{n}$ in $S_i$ that contains $p_i$ in its interior. Moreover we may assume $D_{i,n}$ to be bounded by a piecewise geodesic Jordan curve $J_{i,n}\colon\interval{0}{b_{i,n}}\to X$ \comptwo{p. 1794}{Shi99}{pp. 413-415}{Why35}.\\
Using the Kuratowski embedding we identify $S_i$ with a subset of $l^\infty(X)$. Further we define $\tilde{D}_{i,n}$ to be the union of all linear segments from $p_i$ to a point of $\partial D_{i,n}$. By direct calculation using the embedding we see that linear segments from $p_i$ to distinct points of $S_i$ only intersect in $p_i$. Therefore we get that $F_{i,n}\coloneqq \tilde{D}_{i,n}\cup\p{S_i\setminus D_{i,n}}$ is homeomorphic to $S_i$. Now we equip $F_{i,n}$ with its induced length metric and denote the obtained space by $S_{i,n}$. It follows that the identity map is a homeomorphism between $F_{i,n}$ and $S_{i,n}$. Moreover we have $\conv{S_{i,n}}{S_i}$.\\
For every $\lambda\in\interval[open left]{0}{1}$ the map $\gamma_{i,\lambda}(t)\coloneqq \lambda J_{i,n}\p{\frac{t}{\lambda}}+\p{1-\lambda} p_i$ is a piecewise geodesic Jordan curve in $S_{i,n}$. Moreover the curve bounds a topological closed disc $B_{i,\lambda}$ that contains $p_i$ in its interior. We have that $l_{i,\lambda}\coloneqq \conv{length\p{\gamma_{i,\lambda}}}{0}$, if $\conv{\lambda}{0}$, and there are sequences $\p{\alpha_k}_{k\inN}$ and $\p{\beta_k}_{k\inN}$ in $\interval[open left]{0}{1}$ converging to $0$ such that $l_{1,\alpha_k}=l_{2,\beta_k}$.\\ We assume the subsets in the upcoming construction to be equipped with their induced length metric. There is a natural equivalence relation $\sim$ on $\p{S_{1,n}\setminus \mathring{B}_{1,\alpha_k}}\sqcup\p{S_{2,n}\setminus \mathring{B}_{2,\beta_k}}$ defined by the following condition: $x\sim y$ if there is some $t\in\interval{0}{l_{1,\alpha_{k}}}$ such that $x=\gamma_{1,\alpha_k}(t)$ and $y=\gamma_{2,\beta_k}(t)$. We denote the corresponding metric quotient by $\tilde{S}_{n,k}$ and note that this space is homeomorphic to a connected sum of $S_1$ and $S_2$ \comp{p. 69}{BH99}.\\
Now $\p{\tilde{S}_{n,k}}_{k\inN}$ converges to a metric wedge sum $W_n$ of $S_{1,n}$ and $S_{2,n}$ along $p_1$ and $p_2$. Since $\conv{W_n}{X}$, we may close the proof. \end{proof}
\noindent
From Corollary \ref{cor_spher} and the last lemma we inductively get the following result:
\begin{cor}\label{cor_appr}
Let $X\in\mathcal{G}(c)$. Then there is a sequence of spaces in $\mathcal{S}(c)$ converging to $X$. Moreover the following statements apply:
\begin{itemize}
\item[1)] If all maximal cyclic subsets of $X$ are orientable, then the surfaces of the sequence may be chosen to be orientable.
\item[2)] If there is a non-orientable maximal cyclic subset in $X$, then the surfaces of the sequence may be chosen to be non-orientable.
\end{itemize}
\end{cor}
\noindent
A similar argument to that used in the proof of Lemma \ref{lem_wedge} shows the following result concerning metric 2-point identifications:
\begin{lem}\label{lem_appr2}
Let $S\in\mathcal{S}(c)$ and $X$ be a metric 2-point identification of $S$. Then there is a sequence of spaces in $\mathcal{S}\p{c+2}$ converging to $X$. Moreover the following statements apply:
\begin{itemize}
\item[1)] The surfaces of the sequence may be chosen to be non-orientable.
\item[2)] If $S$ is orientable, then the surfaces of the sequence may be chosen to be orientable.
\end{itemize}
\end{lem}
\noindent
Now we combine the last two results:
\begin{lem}\label{lem_apprG}
Let $Y$ be a space that can be obtained by a successive application of $k>0$ metric 2-point identifications to a space $X\in\mathcal{G}(c)$. Then there is a sequence of spaces in $\mathcal{S}\p{c+2k}$ converging to $Y$. Moreover the following statements apply:
\begin{itemize}
\item[1)] The surfaces of the sequence may be chosen to be non-orientable.
\item[2)] If all maximal cyclic subsets of $X$ are orientable, then the surfaces of the sequence may be chosen to be orientable.
\end{itemize}
\end{lem}
\begin{proof}
The proof proceeds by induction over $k$:\\
In the case $k=0$ the claim directly follows by Corollary \ref{cor_appr}.\\
Now we consider the case $k>0$. Furthermore we assume that the claim is true for every $k_0\inN$ with $k_0<k$. There is a space $Z$ that can be obtained by a successive application of $k-1$ metric 2-point identifications to $X$ such that $Y$ is a metric 2-point identification of $Z$. By the induction hypothesis there is a sequence $\seq{Z}{n}$ in $\mathcal{S}\p{c+2\p{k-1}}$ with $\conv{Z_n}{Z}$ as in the claim.\\
Let $z_1,z_2\in Z$ be the points that are glued to construct $Y$. There is a sequence $\p{z_{i,n}}_{n\inN}$ with $z_{i,n}\in Z_n$ converging to $z_i$. In particular we may assume that $z_{1,n}$ and  $z_{2,n}$ are different. If $Y_n$ denotes a metric 2-point identification of $Z_n$ along $z_{1,n}$ and $z_{2,n}$, then we have $\conv{Y_n}{Y}$.\\ Choosing a diagonal sequence the claim follows by Lemma \ref{lem_appr2}.
\end{proof}
\noindent
Using a metric wedge sum with a vanishing sequence of length spaces all being homeomorphic to the torus or all being homeomorphic to the real projective plane we derive a further corollary of Lemma \ref{lem_wedge}:
\begin{cor}\label{cor_apprS}
Let $S\in\mathcal{S}(c)$. Then the following statements apply:
\begin{itemize}
\item[1)] $S$ can be obtained as the limit of non-orientable closed length surfaces whose connectivity number is equal to $c+1$.
\item[2)] If $S$ is orientable, then $S$ can be obtained as the limit of orientable closed length surfaces whose connectivity number is equal to $c+2$.
\end{itemize}
\end{cor}
\noindent
Now the last three results provide all tools to prove the final result of this section which refines the remaining direction of the \hyperlink{Main Theorem}{Main Theorem}: 
\begin{trm}\label{trm_sec_Main}
Let $c\inN_0$ and $Y$ be space that can be obtained by a successive application of $k$ metric 2-point identifications to a space $X\in\mathcal{G}\p{c_0}$ where $c_0\le c-2k$. Then there is a sequence in $\mathcal{S}(c)$ converging to $Y$. Moreover the following statements apply:
\begin{itemize}
\item[1)] If all maximal cyclic subsets of $X$ are orientable, then the surfaces of the sequence may be chosen to be orientable.
\item[2)] If there is a non-orientable maximal cyclic subset in $X$, $k>0$ or $c_0<c$, then we may assume the surfaces of the sequence to be non-orientable.  
\end{itemize}
\end{trm}
\noindent
We note that the last theorem and Theorem \ref{trm_first_Main} completely describe the Gromov-Hausdorff closure of the class of length spaces being homeomorphic to a fixed closed surface.\\
Finally we are able to prove Corollary \ref{cor_ANR}:
\noindent  
\begin{proof}[Proof of Corollary \ref{cor_ANR}.] 
Let $Y$ be a space that can be obtained by a successive application of metric 2-point identifications to a generalized cactoid. By Theorem \ref{trm_topo} and the \hyperlink{Main Theorem}{Main Theorem} all such spaces are locally simply connected. Hence Proposition \ref{prop_HNN} and Proposition \ref{prop_fund} imply that $Y$ is simply connected if and only if $Y$ is a cactoid.\\
From Theorem \ref{trm_sec_Main} follows that every geodesic cactoid can be obtained as the limit of length spaces being homeomorphic to $S$. Now the claim follows by the \hyperlink{Main Theorem}{Main Theorem} and Corollary \ref{cor_gen_ANR}. 
\end{proof}
\section*{Acknowledgments}
The author thanks his PhD advisor Alexander Lytchak for great support.
\bibliography{sources}{}

\begin{thebibliography}{10}

\bibitem{AKP19}
S.~Alexander, V.~Kapovitch, and A.~Petrunin.
\newblock {\em An invitation to {Alexandrov} geometry. {CAT}(0) spaces}.
\newblock SpringerBriefs Math. Cham: Springer, 2019.

\bibitem{Bab95}
L.~Babai.
\newblock Automorphism groups, isomorphism, reconstruction.
\newblock In {\em Handbook of combinatorics. {Volume} 2}, pages 1447--1540. Amsterdam: Elsevier (North-Holland); Cambridge, MA: MIT Press, 1995.

\bibitem{Bin49}
R.~H. Bing.
\newblock Partitioning a set.
\newblock {\em Bull. Amer. Math. Soc.}, 55:1101--1110, 1949.

\bibitem{BH99}
M.~R. Bridson and A.~Haefliger.
\newblock {\em Metric spaces of non-positive curvature}, volume 319 of {\em Grundlehren Math. Wiss.}
\newblock Berlin: Springer, 1999.

\bibitem{BBI01}
D.~Burago, Y.~Burago, and S.~Ivanov.
\newblock {\em A course in metric geometry}, volume~33 of {\em Grad. Stud. Math.}
\newblock Providence, RI: American Mathematical Society, 2001.

\bibitem{DK18}
C.~Dru{\c{t}}u and M.~Kapovich.
\newblock {\em Geometric group theory}, volume~63 of {\em Amer. Math. Soc. Colloq. Publ.}
\newblock Providence, RI: American Mathematical Society, 2018.

\bibitem{Eps66}
D.~B.~A. Epstein.
\newblock Curves on 2-manifolds and isotopies.
\newblock {\em Acta Math.}, 115:83--107, 1966.

\bibitem{FO95}
S.~C. Ferry and B.~L. Okun.
\newblock Approximating topological metrics by {Riemannian} metrics.
\newblock {\em Proc. Amer. Math. Soc.}, 123(6):1865--1872, 1995.

\bibitem{Fri22}
S.~Friedl.
\newblock Topology.
\newblock Lecture notes, University of Regensburg, 2016-2022. \url{https://friedl.app.uni-regensburg.de/papers/1at-total-public-october-14-2022.pdf}.
\newblock Accessed: 09-06-2023.

\bibitem{Gri54}
H.~B. Griffiths.
\newblock The fundamental group of two spaces with a common point.
\newblock {\em Q. J. Math.}, 5:175--190, 1954.

\bibitem{Gro07}
M.~Gromov.
\newblock {\em Metric structures for {Riemannian} and non-{Riemannian} spaces}.
\newblock Mod. Birkh{\"a}user Class. Boston, MA: Birkh{\"a}user, 2007.

\bibitem{Kaw94}
K.~Kawamura.
\newblock A characterization of {{\(LC^ n\)}} compacta in terms of {Gromov}-{Hausdorff} convergence.
\newblock {\em Canad. Math. Bull.}, 37(4):505--513, 1994.

\bibitem{Loe17}
C.~L{\"o}h.
\newblock {\em Geometric group theory. {An} introduction}.
\newblock Universitext. Cham: Springer, 2017.

\bibitem{LW18}
A.~Lytchak and S.~Wenger.
\newblock Intrinsic structure of minimal discs in metric spaces.
\newblock {\em Geom. Topol.}, 22(1):591--644, 2018.

\bibitem{MST16}
J.~Matou{\v{s}}ek, E.~Sedgwick, M.~Tancer, and U.~Wagner.
\newblock Untangling two systems of noncrossing curves.
\newblock {\em Israel J. Math.}, 212(1):37--79, 2016.

\bibitem{NR23}
D.~Ntalampekos and M.~Romney.
\newblock Polyhedral approximation of metric surfaces and applications to uniformization.
\newblock {\em Duke Math. J.}, 172(9):1673--1734, 2023.

\bibitem{Pet93}
P.~V. Petersen.
\newblock Gromov-{Hausdorff} convergence of metric spaces.
\newblock In {\em Differential geometry: {Riemannian} geometry}, volume 54 part 3 of {\em Proc. Sympos. Pure Math.}, pages 489--504. Providence, RI: American Mathematical Society, 1993.

\bibitem{Pet16}
P.~V. Petersen.
\newblock {\em Riemannian geometry}, volume 171 of {\em Grad. Texts in Math.}
\newblock Cham: Springer, 3rd edition, 2016.

\bibitem{RS38}
J.~H. Roberts and N.~E. Steenrod.
\newblock Monotone transformations of two-dimensional manifolds.
\newblock {\em Ann. of Math. (2)}, 39:851--862, 1938.

\bibitem{Sak13}
K.~Sakai.
\newblock {\em Geometric aspects of general topology}.
\newblock Springer Monogr. Math. Tokyo: Springer, 2013.

\bibitem{Shi99}
T.~Shioya.
\newblock The limit spaces of two-dimensional manifolds with uniformly bounded integral curvature.
\newblock {\em Trans. Amer. Math. Soc.}, 351(5):1765--1801, 1999.

\bibitem{SW01}
C.~Sormani and G.~Wei.
\newblock Hausdorff convergence and universal covers.
\newblock {\em Trans. Amer. Math. Soc.}, 353(9):3585--3602, 2001.

\bibitem{Why35}
G.~T. Whyburn.
\newblock On sequences and limiting sets.
\newblock {\em Fund. Math.}, 25:408--426, 1935.

\bibitem{Why42}
G.~T. Whyburn.
\newblock {\em Analytic topology}, volume~28 of {\em Amer. Math. Soc. Colloq. Publ.}
\newblock American Mathematical Society, Providence, RI, 1942.

\bibitem{You49}
G.~S.~j. Young.
\newblock On 1-regular convergence of sequences of 2-manifolds.
\newblock {\em Amer. J. Math.}, 71:339--348, 1949.

\end{thebibliography}
\bibliographystyle{abbrv}
\ \\
\footnotesize{\textsc{Institute of Algebra and Geometry, Karlsruhe Institute of Technology, Germany.}\\
\emph{E-mail}: \url{tobias.dott@kit.edu}} 
\end{document}